\documentclass{article}

\usepackage{mathtools}
\usepackage{amssymb}
\usepackage{amsthm}
\usepackage{comment}
\usepackage[backend=bibtex,style=alphabetic,citestyle=alphabetic]{biblatex}
\usepackage{mathtools}
\usepackage[title,toc]{appendix}
\usepackage{hyperref}
\usepackage{bbm}
\usepackage[utf8]{inputenc}
\usepackage{tikz-cd}
\usepackage{wasysym}
\usepackage{varwidth}
\usepackage{mathtools}
\usepackage{xspace}
\usepackage{aliascnt}

\makeatletter
\renewcommand*\sectionautorefname{\S\@gobble}
\renewcommand*\subsectionautorefname{\S\@gobble}
\renewcommand*\subsubsectionautorefname{\S\@gobble}
\makeatother
\mathchardef\mhyphen="2D
\DeclareMathOperator{\Hom}{Hom}

\DeclareMathOperator{\FinSet}{\mathbf{FinSet}}

\DeclareMathOperator{\End}{End}
\DeclareMathOperator{\Rep}{Rep}
\newcommand{\Vect}{\mathbf{Vect}}

\DeclareMathOperator{\ind}{Ind}

\DeclareMathOperator{\Id}{Id}

\DeclareMathOperator{\gr}{\mathbf{gr}}

\newcommand{\GL}{GL}
\newcommand{\CCC}{\mathcal{C}}

\newcommand{\CatAd}{\mathbf{2}\mhyphen\Vect_{\gr}}

\newcommand{\Hop}{\mathcal{F}}
\newcommand{\Heis}{\operatorname{Heis}}

\newcommand{\KH}{\mathcal{H}}

\newcommand{\NN}{\mathbbm{N}}
\newcommand{\ZZ}{\mathbbm{Z}}
\newcommand{\QQ}{\mathbbm{Q}}
\newcommand{\one}{\mathbbm{1}}

\newcommand{\kk}{\mathbbm{k}}

\newcommand{\CS}[1]{{\mathcal{C}^{\otimes #1}}}
\newcommand{\AAS}[1]{A^{\otimes #1}}
\newcommand{\PP}{\mathcal{P}}

\newcommand{\HH}{\mathcal{H}}

\newcommand{\EndAd}{\End_{ad}}

\newcommand{\fset}[1]{[#1]}

\newcommand{\monunit}{\mathbbm{1}}

\newcommand{\tik}{\begin{tikzcd}}
\newcommand{\tak}{\end{tikzcd}}

\makeatletter
\def\latearrow#1#2#3#4{%
  \toks@\expandafter{\tikzcd@savedpaths\path[/tikz/commutative diagrams/every arrow,#1]}%
  \global\edef\tikzcd@savedpaths{%
    \the\toks@%
    (\tikzmatrixname-#2)% \noexpand\tikzcd@sourceanchor)%
    to%
    node[/tikz/commutative diagrams/every label] {$#4$}
    (\tikzmatrixname-#3)% \noexpand\tikzcd@targetanchor)
;}}
\makeatother

\def\stik#1#2{
\begin{tikzpicture}[baseline= (a).base]%
\node[scale=#1] (a) at (0,0){%
\begin{tikzcd}[ampersand replacement=\&]%
#2
\end{tikzcd}%
};%
\end{tikzpicture}
}

\def\ttstik#1#2{
\begin{tikzpicture}[baseline= (a).base]%
\node[scale=#1] (a) at (0,0){%
\begin{tikzcd}[ampersand replacement=\&,column sep=tiny]%
#2
\end{tikzcd}%
};%
\end{tikzpicture}
}

\def\nttstik#1#2{
\begin{tikzpicture}[baseline= (a).base]%
\node[scale=#1] (a) at (0,0){%
\begin{tikzcd}[ampersand replacement=\&,column sep=tiny,row sep=small]%
#2
\end{tikzcd}%
};%
\end{tikzpicture}
}

\newcommand{\mapstack}[2]{
\begin{tikzpicture}[anchor=base, baseline,inner sep=0, row sep=0]%
\node[scale=0.6] (b) at (0,0.3){
$#1$
};%
\node[scale=0.6] (a) at (0,0){%
$#2$
};%
\end{tikzpicture}
}

\def\equationautorefname~#1\null{(#1)\null}
\newtheorem{Theorem}{Theorem}

\newtheorem{Proposition}{Proposition}[section]

\newaliascnt{Corollary}{Proposition}
\newtheorem{Corollary}[Corollary]{Corollary}
\aliascntresetthe{Corollary}

\newaliascnt{Lemma}{Proposition}
\newtheorem{Lemma}[Lemma]{Lemma}
\aliascntresetthe{Lemma}

\theoremstyle{definition}
\newaliascnt{Definition}{Proposition}
\newtheorem{Definition}[Definition]{Definition}
\aliascntresetthe{Definition}

\theoremstyle{remark}
\newaliascnt{Notation}{Proposition}

\aliascntresetthe{Notation}

\newaliascnt{Remark}{Proposition}
\newtheorem{Remark}[Remark]{Remark}
\aliascntresetthe{Remark}

\newaliascnt{Example}{Proposition}
\newtheorem{Example}[Example]{Example}
\aliascntresetthe{Example}

\begin{document}

%+Title
\title{Hopf categories and the categorification of the Heisenberg algebra via graphical calculus}
\author{Elena Gal}
\date{\today}
\maketitle
%-Title

%+Abstract
\begin{abstract}
   We explore the connection between the notion of Hopf category and the categorification of the infinite dimensional Heisenberg algebra via graphical calculus proposed by M.Khovanov. We show that the existence of a Hopf structure on a semisimple symmetric monoidal abelian category implies existence of a categorical action in the sense of Khovanov and thus leads to a strong categorification of this algebra. 
   \end{abstract}
%-Abstract

\numberwithin{equation}{section}
%+Contents
\tableofcontents
%-Contents
\section{Introduction}

The classical one-variable Heisenberg algebra is the $\ZZ$-algebra with two generators $p,q$ and one defining relation $[p,q]=1$. There are several infinite-dimensional generalizations that appear naturally in different settings. The most straightforward one is to consider two infinite families of generators $p_n, q_n, n\in \NN$ with the relations 
\begin{itemize}
\item
$p_0=q_0=1$
\item 
$[p_m,p_n]=[q_m,q_n]=0$
\item
$[p_m,q_n]=\delta_{mn}1$
\end{itemize} 
This algebra is a simplest example of a vertex operator algebra. It has a unique faithful irreducible representation called the Fock space.

Khovanov in \cite{Khov} proposed a categorification of a $\ZZ$-form of this algebra: the algebra $H$ generated over $\ZZ$ by two families of elements $a_n,b_n,n\in\NN$ where $a_0=b_0=1$, with $a_i, a_j$ and $b_i,b_j$ commuting for all $i,j \in \NN$ and with the relations 
\begin{equation}
\label{HeisRel}
[a_m,b_n]=b_{n-1}a_{m-1}
\end{equation}

He constructs a semisimple abelian category $\KH'$ whose Karoubi envelope $\KH$ is equipped with a natural embedding 
\begin{equation}
%\label{HeisIso}
\gamma:H\rightarrow K_0(\KH)
\end{equation}
defined by sending the generators of $H$ into certain elements of $K_0(\KH)$ and showing that the relations \eqref{HeisRel} hold. It is conjectured in \cite{Khov} that $\gamma$ is an isomorphism.

Using this category one can think of a categorical Fock space action as a functor from $\KH'$ to the category of endofunctors of an abelian category. An example of such action can be found in \cite{hongyacobiPoly} for the category $\PP$ of polynomial functors.

The categorification in \cite{Khov} is given in the form of graphical calculus. One downside of this approach is that although the graphical calculus for the category $\KH'$ are described explicitly, the graphical calculus for its Karoubi envelope $\KH$ are complicated to formulate and track. Therefore it is interesting to approach the definition of the Heisenberg algebra and its categorification from a different point of view. Namely, the infinite dimensional Heisenberg algebra and its Fock space action can be described without the use of generators using the Heisenberg double construction. 

In general, the Heisenberg double construction canonically associates an algebra $\Heis(A,A')$ to a dual pair of Hopf algebras $A,A'$. It comes with a canonical embedding into the algebra $\End(A)$ and hence a canonical faithful representation on the space $A$.
The Heisenberg algebra is obtained by taking $A$ and $A'$ to be the Hopf algebra $\Lambda$ of symmetric functions. The duality is given by the inner product on $\Lambda$, induced by considering the orthogonal basis of Schur polynomials. The multiplication of polynomials and its dual via this product give $\Lambda$ the structure of a self-dual (or \emph{selfadjoint}) Hopf algebra. The structure of self-dual Hopf algebras, and in particular different generator systems for $\Lambda$ were explored by Zelevinsky in the book \cite{Zelbook}.

A nice feature of this generator-less approach is that it produces all different $\ZZ$-forms of Heisenberg algebra by picking different elements as generators. For example the $\ZZ$-form categorified by Khovanov is obtained by setting 
$a_n$ to be the elementary symmetric function
of degree $n$ in the left $\Lambda$ and the $b_n$ to be the complete homogeneous symmetric
function of degree $n$ ($\sum_{i_1\leq\dots\leq i_n} x_{i_1}\cdots x_{i_n}$) in the right $\Lambda$. The $\ZZ$-form mentioned in the beginning of this introduction is obtained by taking $p_n$ to be elementary symmetric function of degree $n$ ($\sum_{i_1<\dots<i_n} x_{i_1}\cdots x_{i_n}$) in the left $\Lambda$ and letting the $q_n$ to be the primitive symmetric function of degree $n$ (see \cite{Zelbook}) in the right $\Lambda$. The algebra generated by these elements is isomorphic to $H$ over $\QQ$ but not over $\ZZ$. 
 
This approach to the Heisenberg algebra suggests an alternative idea for its categorification. 
In fact, the Hopf algebra structure on $\Lambda$ is induced by the isomorphism of algebras $\Lambda \cong \bigoplus_{n\geq 0} K(\Rep S_n)$. The multiplication in $\bigoplus_{n\geq 0} K(\Rep S_n)$ corresponds to the morphism induced by the induction functor and the inner product is given by dimension of the $\Hom$ spaces. Hence $\Lambda$ is a decategorification of a higher object, and so is the Heisenberg algebra. Formulating a working  definition for a higher Hopf algebra structure and using it to repeat the construction of Heisenberg double on the categorical level would give us a categorification of the Fock space.
 
In \cite{mySSH}(joint with A.Gal) we proposed a notion of selfadjoint Hopf structure for semisimple abelian categories. Using this structure we construct in \S 6.2 Theorem 2 of \cite{mySSH} an isomorphism categorifying the canonical relation defining the algebra structure on the Heisenberg double. We recall this construction here in \autoref{th:deltam} of \autoref{CatHeis}.
In particular we consider selfadjoint Hopf structure on the category $\PP$ of polynomial functors whose $K$-group (with the induced Hopf algebra structure) is isomorphic to $\Lambda$. Then this isomorphism categorifies the relations \eqref{HeisRel} between the generators of the algebra $H$.
Thus we obtain a categorification of the Fock space action of $H$ constructed from the Hopf category structure on the categorification of the Fock space $\Lambda$ itself. 

It is then natural to ask what is the relation between this categorification and the one proposed in \cite{Khov}. We explore this question in the present work. In \autoref{KhovCalculus} we show that the categorical Heisenberg action in the sense of \cite{mySSH} implies the existence of the functor 
\[F:\HH \rightarrow \End{\CCC}\] 
under an assumption on the adjunction data providing the selfadjoint Hopf structure on $\CCC$. Thus the Heisenberg categorification of \cite{mySSH} arising the Hopf category structure on a category $\CCC$ is in fact a strong categorification.

A nice property of the categorical Heisenberg double isomorphism from \autoref{th:deltam} is that it categorifies the defining relations of the Heisenberg algebra explicitly for any choice of elements. This is in contrast to categorifications using graphical calculus which are much less explicit due in particular to the requirement of working in the Karoubi envelope.  Philosophically speaking this is the consequence of the categorical Heisenberg double structure being the derivation of a relatively simple structure - the Hopf category structure on the underlying categorified Fock space. It is our hope that this fact will prove useful in working with more complicated vertex algebras, for example to categorify the Boson-Fermion correspondence.

The generators of Khovanov's category $\KH$ correspond in our construction to the adjoint functors of multiplication and comultiplication by an element of the Fock space. In principle for a selfadjoint Hopf category we can have two functors of comultiplication, corresponding to left and right adjoints of multiplication. They play different roles in the  categorification of the Heisenberg double relation from \autoref{th:deltam}. To reconstruct Khovanov's graphical calculus we must assume these adjoints to be equal.However there are interesting examples of Heisenberg doubles where they shouldn't be (see for example \cite{savageyacobi2015}) and it would be interesting to explore the categorification that our approach provides in those cases.

Another interesting direction is in considering the Heisenberg categorification arising from the braided monoidal category with Hopf structure, specifically the category $\bigoplus_{n\geq 0} \Rep \GL_n(\mathbb{F}_q)$. It is unclear what graphical calculus correspond to such categorification; however the generalization of the Heisenberg double construction seems to be quite straightforward. We expect some categorification of Hecke relations to arise as a result in \cite{BraidingWaldhausen}.

Finally, considering the Heisenberg double construction in the finite characteristic case should provide insights into the theory of modular representations of symmetric groups. 
%We apply the above observation to the Heisenberg action on the category selfadjoint Hopf category $\PP$ to give a short proof that the morphism \eqref{HeisIso} is in fact a morphism of algebras. In \cite{Khov} a proof of this fact is indicated using a generalization of the graphical calculus to work with the objects of the Karoubi envelope $\KH$. In this work the proof is an \autoref{MorphismofAlgebras} of \autoref{}.

\subsection{Notations}
\begin{itemize}
\item $\kk$ - a field of characteristic 0
\item $\Vect$ - the category of finite dimensional vector spaces over $\kk$
%\item $S,T,\ldots$ - finite sets
\item $\fset{n}$ - the set $\{1,2,\ldots,n\}$
\end{itemize}
\section{Hopf categories and a categorical Heisenberg double}
\label{HopfHeis}
In this section we review the notions of the symmetric selfadjoint Hopf category and the associated categorical Heisenberg double constructed in \cite{mySSH}.
\subsection{Self-adjoint Hopf categories}
\label{HopfCat}
The symmetric selfadjoint Hopf (SSH) category is a categorification of the notion of positive selfadjoint Hopf (PSH) algebra introduced by A.Zelevinsky in \cite{Zelbook}. 
\begin{Definition}
\label{def:PSH}
A PSH algebra is a graded connected Hopf $\ZZ$-algebra with an inner product and a distinguished finite orthogonal $\ZZ$ basis in each grade such that multiplication and comultiplication are adjoint and positive maps.
\end{Definition}
The examples of such algebras naturally appear as $K$-groups of categories, hence one would like to define a similar notion on the level of categories. The idea for the definition of the SSH category comes from the observation that the Hopf condition can be visualized as the diagram 
\begin{equation}
\label{Hopfsquare}
\stik{1}{
\AAS{4} \arrow[]{r}{\overline m} \arrow[leftarrow]{d}[left]{\Delta^{\otimes 2}}\& \AAS{2} \arrow[leftarrow]{d}{\Delta}\\
\AAS{2} \arrow[]{r}{m} \& A}
\end{equation}
where
\begin{align*}
\overline{m}(x\otimes y\otimes z \otimes w)&=m(x\otimes z)\otimes m(y \otimes w)\\
\Delta^{\otimes 2}(x \otimes y)&=\Delta(x)\otimes \Delta(y)
\end{align*}
In the case of the PSH algebra this square can be obtained from the commutative square of multiplications
\begin{equation}
\label{Hopfsquaremult}
\stik{1}{} \stik{1}{
\AAS{4} \arrow[]{r}{\overline m} \arrow[]{d}[left]{m^{\otimes 2}}\& \AAS{2} \arrow[]{d}{m}\\
\AAS{2} \arrow[]{r}{m} \& A
}
\end{equation}
by replacing the verticals with their adjoints.

We utilize this observation as follows. On the categorical level we should require commutativity of the square \eqref{Hopfsquare} up to isomorphism. Such isomorphism will have to satisfy coherence relations that in general are difficult to specify and track. Instead we will consider it as property of the square of the form \eqref{Hopfsquaremult}. On the categorical level we will need to replace tensor of algebras with Deligne tensor of categories and maps with functors. We will also assume that our square commutes up to an isomorphism $\alpha$.

The procedure of of replacing two opposite sides of the square of functors with adjoints is called taking the \emph{mate} of the square. This is the basic tool we are going to use through the article to construct categorification of Heisenberg algebra.

\begin{Definition}
\label{mate}
Given a square
\begin{equation}
\label{bcsquarebefore}
\tik
A\arrow[]{d}{f} \arrow[]{r}{g}& B \arrow[]{d}{h} \arrow[shorten >=0.4cm,shorten <=0.4cm,Rightarrow]{dl}[above,sloped]{\alpha} \\ 
C \arrow[]{r}{i} & D
\tak
\end{equation}
where $\alpha:h\circ g\rightarrow i\circ f$ is a (not necessarily invertible) 2-morphism we obtain by replacing the the verticals $h,f$ with their right adjoints $h_R,f_R$  the \emph{right mate} of the above square
\[
\tik
A \arrow[]{r}{g}\arrow[blue,shorten >=0.4cm,shorten <=0.4cm,Rightarrow]{dr}[above,sloped]{\alpha_R}& B   \\ 
C \arrow[dashed,blue,thick]{u}{f_R} \arrow[]{r}{i} & D \arrow[dashed,blue,thick]{u}[right]{h_R}
\tak
\]
where $\alpha_R$ is given by
\[
g\circ f_R\rightarrow h_R\circ h \circ g \circ f_R \xrightarrow{\alpha} h_R \circ i \circ f \circ f_R \rightarrow h_R \circ i
\]
\end{Definition}
\begin{Definition}
A square as in \eqref{bcsquarebefore} with 2-morphism $\alpha$ is said to satisfy the right \emph{Beck-Chevalley condition} if $\alpha_R$ is invertible.
\end{Definition}

\begin{Remark}
The left mate and Beck-Chevalley condition is defined in a similar way.
\end{Remark}

Our next observation is that the square \eqref{Hopfsquaremult} is related to the Cartesian square of finite sets 
\begin{align*}
\stik{1}{
\AAS{4} \arrow[]{r}{\overline m} \arrow[]{d}[left]{m^{\otimes 2}} \& \AAS{2} \arrow[]{d}{m}\\
\AAS{2} \arrow[]{r}{m} \& A
}
& {} & \leftarrow & {} &
\stik{1}{
\fset{4} \arrow[]{r}{} \arrow[]{d}[left]{} \& \fset{2} \arrow[]{d}{} \ar[draw=none]{dl}[above,yshift=-0.5pc,xshift=-0.8pc,scale=2]{\ulcorner}\\
\fset{2} \arrow[]{r}{} \& \fset{1}
}
\end{align*}

In fact, as explained in \S 3.1 of \cite{mySSH} any Cartesian square of finite sets should correspond to a square of multiplications satisfying the Beck-Chevalley condition in a categorification of Hopf algebra. Altogether this leads to the following definition of symmetric selfadjoint Hopf category:

\begin{Definition}
\label{HopfCategory}
A symmetric selfadjoint Hopf structure for a semisimple abelian $k$-linear category $\CCC$ is given by a monoidal functor of 2-categories
\[\FinSet\xrightarrow{\Hop}\CatAd\]
which takes Cartesian squares in $\FinSet$ to squares satisfying the Beck-Chevalley condition and such that $\Hop(\fset{1})=\CCC$.
\end{Definition}
Here $\CatAd$ has objects graded 2-vector spaces as in \cite{2Vect}, i.e. semisimple $\kk$-categories equivalent to a finite sum of copies of $\Vect$. The 1-morphisms are finite sums of exact morphisms of bounded degree and the 2-morphism are the natural transformations between those. The monoidal structure is given by the Deligne tensor of categories.
\begin{Remark}
1-morphisms in $\CatAd$ admit left and right adjoints.
\end{Remark}
In particular \autoref{HopfCategory} implies that we have the functor $m:\CS{2} \rightarrow \CCC$ as an image of the map of finite sets $\fset{2}\rightarrow \fset{1}$ and its left and right adjoints $\Delta^l,\Delta^r: \CS{2} \rightarrow \CCC$. The naive categorical version of the bialgebra condition is a consequence of the image
\begin{equation}
\label{HopfRel}
\tik
\CS{4} \arrow[]{r}{\overline m} \arrow[]{d}[left]{m^2}& \CS{2} \arrow[]{d}{m}\arrow[shorten >=0.4cm,shorten <=0.4cm,Rightarrow]{dl}[above,sloped]{\sim}\\
\CS{2} \arrow[]{r}{m} & \CCC
\tak
\end{equation}
of the Cartesian square 
\[
\stik{1}{
\fset{4} \arrow[]{r}{} \arrow[]{d}[left]{} \& \fset{2} \arrow[]{d}{} \ar[draw=none]{dl}[above,yshift=-0.5pc,xshift=-0.8pc,scale=2]{\ulcorner}\\
\fset{2} \arrow[]{r}{} \& \fset{1}
}
\]
satisfying the Beck-Chevalley condition, i.e. the square commutes up to invertible morphism when the horizontal or vertical sides are replaces by left(right)adjoints:

Here $\overline m$ corresponds to the maps of sets $\mapstack{1\mapsto 1}{3\mapsto 1},\mapstack{2\mapsto 2}{4\mapsto 2}$.
\begin{Remark}
\label{SSHbraiding}
The restriction of the 2-morphism \autoref{HopfRel} provides a braiding isomorphism for $\CCC$, and as checked in Proposition 3.18 of \cite{mySSH} this provides $\CCC$ with symmetric monoidal structure.
\end{Remark}
\begin{Remark}
\label{SSHunit}
By Proposition 3.17 of \cite{mySSH} 
$\Hop(\emptyset)$ is equivalent to $\Vect$. The image of the map $\emptyset\rightarrow 1$ in $\FinSet$ equips an SSH category $\CCC$ with a functor $\Vect \xrightarrow{m_{\emptyset}}\CCC$. $m_{\emptyset}(\kk) \in \CCC$ categorifies the notion of the unit of the Hopf algebra.
\end{Remark}

A simplest example of an SSH category is the category of polynomial functors $\PP$ from \cite{FrieSusPoly}. \begin{Definition}
The category of polynomial functors $\PP$ is the category whose objects are functors from $\Vect$ to $\Vect$ that induce polynomial maps of the $\Hom$ spaces.
\end{Definition}
The SSH structure on $\PP$ is described in \S 4 of \cite{mySSH}. To give a simple description illustrating our abstract definition, the Deligne tensor $\PP^{\otimes2}$ can be represented by a category of bi-polynomial functors, i.e. the functors 
\[\Vect\times\Vect \rightarrow\Vect\] Then we have 
\begin{align*}
&m:\PP^{\otimes 2}\rightarrow\PP, 
& m\Phi(V)&=\Phi(V,V) &\Delta: \PP\rightarrow \PP^{\otimes 2}, \Delta F(V,W)=F(V\oplus W)
\end{align*}
and the unit is represented by the polynomial functor $\one:\kk\rightarrow V$. From this examples it can be easily seen how to construct an image of any map of finite sets and its adjoint.

The SSH category $\PP$ categorifies the PSH algebra $\Lambda$ of symmetric functions.
\subsection{Heisenberg double}
\label{Heisenbergdouble}
A dual pair of graded Hopf algebras $A,A'$ gives rise to an algebra structure on the space $A\otimes A'$.  This algebra is called the \emph{Heisenberg double}. For our purposes we will view this algebra structure as induced from the algebra structure on $\End_{\ZZ}(A)$ via the injective morphism of ${\ZZ}$-modules $A\otimes (A')^{op} \rightarrow\End_{\ZZ}(A)$. This morphism is given by considering the action of $A$ on itself by left multiplication by an element and the action of $A'$ on $A$ adjoint to the right multiplication by an element.

Let us consider this situation in more detail. We consider $A=
\bigoplus_{n\geq 0}A_n$, $A'=\bigoplus_{n\geq 0}A'_n$ with components free and finitely generated over $\ZZ$. Let $<,>$ be a pairing between these algebras that induces isomorphism $A'\cong\bigoplus_{n\geq 0}A_n^*$ and respects the Hopf structures. That is $<,>$ satisfies
\begin{gather*}
\left<m(a_1\otimes a_2),b\right>=\left<a_1\otimes a_2,\Delta ' b\right>\\
\left<\Delta a,b_1\otimes b_2\right>=\left<a,m'(b_1\otimes b_2)\right>\\
\left<a,1_{A'}\right>=\varepsilon(a), \left<1_{A},b\right>=\varepsilon(b)
\end{gather*}
where $m,m',\Delta,\Delta'$ are the morphisms of multiplication and comultiplication in $A$ and $A'$:
\begin{align*}
m: A\otimes A &\rightarrow A &\Delta:A&\rightarrow A\otimes A\\
m': A'\otimes A' &\rightarrow A' &\Delta':A'&\rightarrow A'\otimes A'
\end{align*}
For each $x\in A$ we can then consider the operator $m_x: A\rightarrow A$ of left multiplication by an element $x \in A$. We also consider the operator $\Delta_y: A\rightarrow A, y\in A'$. $\Delta_y$ is adjoint to the operator of right multiplication by $y, m'_y:A'\rightarrow A'$ via the pairing $<,>$.

The assigments $x \rightarrow m_x$, $y \rightarrow \Delta_y$ define morphisms of algebras $A, (A')^{op} \rightarrow \End_{\ZZ}{A}$. We can also use these operators to define the following morphism of $\ZZ$-modules: 
\begin{align*}
\varphi:A\otimes A'^{op} &\rightarrow\End_\ZZ(A) &x\otimes y &\mapsto m_x\Delta_y
\end{align*}
\begin{Proposition}[see e.g. \cite{mySSH} Proposition 5.4]
\label{prop:heisdouble}
$\varphi$ is injective and its image is a subalgebra of $End_{\ZZ}(A)$
\end{Proposition}

Since $\varphi$ is injective, it induces an algebra structure on $A\otimes A'$. This algebra is called \emph{the Heisenberg double}. It is equipped with a canonical action on $A$ defined by $\varphi$. The analog of Stone-von Neumann theorem for this action is considered in \cite{savageyacobi}.

A special case of this construction is a Heisenberg double associated to a PSH algebra of \cite{Zelbook}. In the case of the simplest PSH algebra $\Lambda$ of symmetric functions this construction recovers the infinite dimensional Heisenberg algebra. More precisely to get the particular $\ZZ$- form of Heisenberg algebra one should choose a suitable set of generators in the Heisenberg double associated to $\Lambda$. To recover the form given by the relations \autoref{HeisRel} which was categorified by Khovanov in \cite{Khov} one should consider series of generators for $n \geq 0$: $a_n$ in the left $\Lambda$ and $b_n$ in the right $\Lambda$ with $a_n$ being the elementary symmetric function
of degree $n$ and the $b_n$ the complete homogeneous symmetric function of degree $n$, i.e. $b_n=\sum_{i_1\leq\dots\leq i_n} x_{i_1}\cdots x_{i_n}$. In this article we will focus on categorification of Heisenberg doubles associated to PSH algebras. We will denote Heisenberg double associated to a PSH algebra $A$ by $\Heis{A}$

To explain how to approach the notion of Heisenberg double using SSH categories we briefly recall the key idea of the proof of \autoref{prop:heisdouble}. To prove that the image of $\varphi$ is a subalgebra we show that the following relation holds $\forall x\in A'$
\begin{align}
\label{eqn:deltam}
\Delta_xm = m\Delta^2_{\Delta'(x)}
\end{align}
where if $\Delta'(x)=x_{(1)}\otimes x_{(2)}$ (in Sweedler notation) then 
\begin{align}
\label{DeltaSquared}
\Delta^2_{\Delta'(x)}(y\otimes z):=\Delta_{x_{(1)}}y\otimes\Delta_{x_{(2)}}z
\end{align}
This implies that the image of $\varphi$ is closed on multiplication in the algebra $\End_\ZZ(A)$ as $\forall x \in A, y\in A'$:
\begin{align*}
\Delta_xm_y &= m\Delta^2_{\Delta'(x)}i_y=m_{\Delta_{x_{(2)}}(y)}\Delta_{x_{(1)}}\in\varphi(A\otimes A') 
\end{align*}
where 
\[i_y(a)=y\otimes a\]

\begin{Remark}
The relation $\Delta_xm = m\Delta^2_{\Delta(x)}$ can be interpreted as a statement about commutativity of a square of morphisms: 
\begin{equation}
\tik
A^{\otimes 2} \arrow[]{r}[above]{m}   & A   \\
A^{\otimes 2} \arrow[]{u}{(\Delta)^2_{\Delta'(x)}} \arrow[]{r}[below]{m} & A \arrow[]{u}[right]{\Delta_x}
\tak
\end{equation}
This suggests an approach to categorification of Fock space action using a Hopf category structure and the Beck-Chevalley condition that we describe in the next section.
\end{Remark}
%%%%%%%%%%%%%%%%%%%%%%%%%%%%%%%%%%%%%%%%%%%%%%%%%%%%%%%%%%%%%%%%%%%%%%%%%%%%%%%
%%%%%%%%%%%%%%%%%%%%%%%%%%%%%%%%%%%%%%%%%%%%%%%%%%%%%%%%%%%%%%%%%%%%%%%%%%%%%%%
\subsection{Categorical Heisenberg double}
\label{CatHeis}
%Recap of Heisenberg doubles from Hopf, SSH and categorical Heisenberg double with detailed description of the construction of the lift of Heisenberg relation
%$\CCC$ $k$-linear semisimple abelian category. Introduce $m$,$\Delta$ 
In this section we will construct a categorification of the notion of Heisenberg double and the associated Fock space for a symmetric selfadjoint Hopf category. Consider the $k$-linear semisimple abelian category $\CCC$ with Hopf structure given by \autoref{HopfCategory}. Then in particular we have the functor $m:\CS{2} \rightarrow \CCC$ and its left and right adjoints $\Delta^l,\Delta^r: \CS{2} \rightarrow \CCC$.

Consider the functor $\boxtimes:\CCC^{\times 2}\rightarrow\CS{2}$ which is part of the definition of the Deligne tensor. For any object $F\in\CCC$ define the functor 
\begin{align*}
i_F:\CCC & \rightarrow\CS{2}, & i_F(X)=F\boxtimes X
\end{align*}
Note that $i_F$ has left and right adjoints which we denote by $j^l_F,j^r_F$
We define $m_F,\Delta^l_F,\Delta^r_F$ as compositions
\begin{align*}
m_{F}&=m\circ i_F & \Delta^r_F&=j^r_F\circ\Delta^r & \Delta^l_F&=j^l_F\circ\Delta^l 
\end{align*} 
$\Delta^r_F, \Delta^l_F$ are respectively right and left adjoint to $m_F$ for any $F \in \CCC$. 

In a similar fashion, using the fact that Deligne tensor defines a symmetric monoidal structure on the 2-category $\CatAd$,
we can define adjoint functors $\Delta^2_{\Delta F}\dashv m^2_{\Delta^r F}\dashv \Delta^2_{\Delta F}$ for $F \in \CCC$ categorifying the map \autoref{DeltaSquared} we used to define Heisenberg double for Hopf algebras.

%functors of the form $i_{F_1\boxtimes F_2}:\CCC^{\times 2}\rightarrow\CS{4}, i_{F_1\boxtimes F_2}{G_1 \boxtimes G_2}= F_1\boxtimes G_1\boxtimes F_2\boxtimes G_2$ give rise to a functors  $i_\Phi:\CS{2}\rightarrow \CS{4}$ where $\Phi\in\CS{2}$.

Our categorification of the notion of Heisenberg double is provided by the following:

\begin{Theorem}[\cite{mySSH}]
\label{th:deltam}
There is a canonical isomorphism 
\begin{equation}
\label{deltam}
\Delta^r_F m\cong m\circ (\Delta^r)^2_{\Delta^l(F)}
\end{equation}
coming from the SSH structure on $\CCC$.
\end{Theorem}

\begin{Corollary}
We have a canonical isomorphism \[
\Delta^r_Fm_G=\Delta^r_F m\circ i_G \cong m\circ (\Delta^r)^2_{\Delta^l(F)}\circ i_G
\]
\end{Corollary}

Since as shown in \cite{mySSH} the $K$-group of $\CCC$ is a positive selfadjoint Hopf algebra, we can associate to it a Heisenberg double and its Fock space. Then \autoref{th:deltam} implies the following:
\begin{Corollary}
\label{Th:HeisCat}
Let $\CCC$ be an SSH category and denote by $\EndAd(\CCC)$ the category of endofunctors of $\CCC$ admitting adjoints.\\
The functor $\CCC\otimes\CCC^{op}\rightarrow \EndAd(\CCC)$ given by $F\boxtimes G\mapsto m_F\circ \Delta^r_G$ can be naturally constructed using the SSH structure on $\CCC$ and descends to the Fock space representation of the Heisenberg double associated to $(K(\CCC))$.
\end{Corollary}

The proof of \autoref{th:deltam} is based on using the notion of \emph{mate} and the Beck-Chevalley condition recalled in \autoref{HopfCategory}. Namely the isomorphism in the theorem can be represented as a square of functors
\begin{equation}
\label{heisbcsquare}
\tik
\CS{2} \arrow[]{r}[above]{m}  \arrow[shorten >=0.5cm,shorten <=0.5cm,Rightarrow]{dr}[above,sloped]{\sim} & \CCC   \\
\CS{2} \arrow[]{u}{(\Delta^r)^2_{\Delta^l(F)}} \arrow[]{r}[below]{m} & \CCC \arrow[]{u}[right]{\Delta^r_F}
\tak
\end{equation}

That is the right mate of the square 
\begin{equation}
\label{heisquare}
\tik
\CS{2} \arrow[]{d}[left]{m^2_{\Delta^l(F)}} \arrow[]{r}[above]{m} & \CCC \arrow[shorten >=0.4cm,shorten <=0.4cm,Rightarrow]{dl}[above,sloped]{\alpha}  \arrow[]{d}[right]{m_F}\\
\CS{2} \arrow[]{r}[below]{m}& \CCC 
\tak
\end{equation}

The above square can be seen as a composition of the squares 
\begin{equation}
\label{HeisDecomp}
\tik
\CS{2} \arrow[]{d}[left]{i_{\Delta^l(F)}} \arrow[]{r}{m} & \CCC \arrow[shorten >=0.4cm,shorten <=0.4cm,Rightarrow]{dl}[above,sloped]{}  \arrow[]{d}{i_F}\\
\CS{4} \arrow[]{r}{\overline m} \arrow[]{d}[left]{m^2}& \CS{2} \arrow[]{d}{m}\arrow[shorten >=0.4cm,shorten <=0.4cm,Rightarrow]{dl}[above,sloped]{\sim}\\
\CS{2} \arrow[]{r}{m} & \CCC
\tak
\end{equation}
where the bottom square is \autoref{HopfRel}, the the part of the SSH structure categorifying the bialgebra condition. Hence it satisfies the Beck-Chevalley condition. The 2-morphism in the upper square is constructed using the unit of the adjunction $\eta_L: id \rightarrow m\Delta$. Using this fact one can check that it satisfies the Beck-Chevalley condition as well. 

A general method for constructing morphisms using the SSH structure and higher category theory is outlined in \S 6.3 of \cite{mySSH}.

\begin{Remark}
\label{rem:leftright}
The isomorphism $\beta$ is constructed by taking the right mate. Hence to construct it we have to use both the left and right adjoints of the multiplication. To reconstruct Khovanov's categorification of the Heisenberg algebra of \autoref{KhovCalculus} we must assume these adjoints to be equal. However this might prove useful for categorifications of Heisenberg doubles in the cases when they are not.
\end{Remark}

%In \autoref{ssec:FockP} we will apply our general construction to the case of the SSH category $\PP$ of polynomial functors and explicitly describe the $1$- and $2$-morphisms given by it in this example. 

\section{Heisenberg categorification via the graphical calculus}
\label{KhovanovHeisenberg}
In this chapter we recall the definition of the Heisenberg category $\KH$ from \cite{Khov}. On the algebra level we start from the algebra generated over $\ZZ$ by two families of elements $a_n,b_n,n\in\NN$ where $a_0=b_0=1$, with $a_i, a_j$ and $b_i,b_j$ commuting for all $i,j \in \NN$ and with the relations 
\begin{equation*}
[a_m,b_n]=b_{n-1}a_{m-1}
\end{equation*}
As mentioned in \autoref{Heisenbergdouble} via the canonical embedding of the Heisenberg double $\Heis{\Lambda}$ into $\End_{\ZZ}(\Lambda)$ these generators can be thought of as the operators $m_{a_i}, \Delta_{b_i}$, where $a_i$ are elementary and $b_i$ complete homogeneous symmetric functions. Under the isomorphism $\Lambda\cong \bigoplus_n K(\Rep(S_n))$ $a_i$ and $b_i$ correspond to the classes of representations of $S_i$ which we will denote $\Lambda^i$ and $S^i$ (those are respectively the sign and trivial representations). Note that $a_1=b_1$ go under this isomorphism to the identity representation of $S_1$. The multiplication and comultiplication go to endomorphisms of $\bigoplus K(\Rep(S_n))$ induced by the functors of induction and restriction. 

To construct a categorification of this algebra start from considering a monoidal category $\KH'$ whose objects are sequences of $+$'s and $-$'s, with product given by concatenation of sequences. On the level of $K$-groups $+$ corresponds to multiplication by $a_1$  and $-$ to comultiplication by $b_1$ and the concatenation of sequences to composition.

The spaces of morphisms between these objects are constructed by looking at the categories $\Rep(S_n)$ for all $n$, since the objects of $\KH'$ can be naturally seen as corresponding to compositions of induction and restriction functors from $\Rep(S_n\times S_1))$ to $\Rep(S_{n+1})$ for some $n$. 

These spaces can be described using the graphical calculus for the biadjoint functors modulo relations: the $1$-morphisms in $\KH'$ are the spaces of oriented one-dimensional manifolds up to boundary isotopies and a set of relations. Start from the identity morphisms: they are drawn as lines
\[
\stik{1}{- \& +\ar[leftarrow]{d}  \\- \ar[leftarrow]{u} \& +} 
\]

Generally, the diagrams should be understood as morphisms from the bottom boundary to the top boundary, and are composed accordingly. If the source or the target are unit object of $\KH'$ the corresponding diagram has no lower or upper boundary. 

The two pairs of unit and counit of adjunctions are drawn as half-circles
\begin{equation}
\label{diag:unitcounit}
\stik{1}{- \ar[bend right=90]{r} \& + \& + \ar[bend left=90]{r}  \& - \&,\& - \& + \ar[bend right=90]{l}\& + \& - \ar[bend left=90]{l}}
\end{equation}
We also have the upward pointing "cross" diagram  
\begin{equation}
\label{diag:ucross}
\stik{1}{+ \& + \\ + \ar{ur} \& +  \ar{ul}} 
\end{equation}
which denotes the natural endomorphism of $\ind_{S^n\times S^1\times S^1}^{S^{n+2}}(\rho,\Lambda^1,\Lambda^1)$ given by transposing the two copies of $\Lambda^1$. 
\begin{Remark}
\label{crossmates}
Thinking about the above endomorphism as as a square 
\[
\tik
\CCC \arrow[]{r}{+} \arrow[]{d}[left]{+}& \CCC \arrow[]{d}{+}\arrow[shorten >=0.4cm,shorten <=0.4cm,Rightarrow]{dl}[above,sloped]{\sim}\\
\CCC \arrow[]{r}{+} & \CCC
\tak
\]
the other cross diagrams
\begin{equation}
\label{diag:lrcrosses}
\stik{1}{+ \& - \\ - \ar[leftarrow]{ur} \& +  \ar{ul}} \stik{1}{}
\stik{1}{- \& + \\ + \ar{ur} \& -  \ar[leftarrow]{ul}} \stik{1}{}
\stik{1}{- \& - \\ - \ar[leftarrow]{ur} \& -  \ar[leftarrow]{ul}} 
\end{equation}
are the \emph{mates} (see \autoref{mate}) of the upward facing cross in various directions. This is because they are the result of composing the upward cross with the approporiate unit and counit diagrams.
\end{Remark}
The space of $1$-morphisms in $\KH'$ generated by the above diagrams modulo isotopies and the following relations:

\begin{equation}
\label{KhovRel1}
\stik{1}{
- \ar[bend right=90]{r} \& + \ar[bend right=90]{l}
}
= \Id_{\monunit}
\stik{1}{}
,
\stik{1}{}
\stik{1}{- \& + \\
+ \ar{ur} \& -  \ar[leftarrow]{ul} \\
- \ar[leftarrow]{ur} \& + \ar{ul}
} 
=
\ttstik{1}{- \& + \\
{}\& {} \\
- \ar[leftarrow]{uu} \& + \ar{uu}}
-
\ttstik{1}{- \ar[bend right=90]{r}\& + \\
{} \& {} \\
-  \& + \ar[bend right=90]{l}
} 
\end{equation}
\begin{equation}
\label{KhovRel2}
\stik{1}{+ \& - \\
- \ar[leftarrow]{ur} \& +  \ar{ul} \\
+ \ar{ur} \& - \ar[leftarrow]{ul}
} 
=
\ttstik{1}{
+ \& - \\
{} \& {} \\
+\ar{uu} \& - \ar[leftarrow]{uu}
}
\stik{1}{}
,
\stik{1}{}
\stik{1}{
{} \& {} \& + \\
- \& + \ar[bend right=90]{l} \& + \ar{u} \\
- \ar[bend right=90]{r} \ar[leftarrow]{u} \& + \ar{ur} \& + \ar{ul}\\
{} \& {} \& + \ar{u}
}
=0
\end{equation}
\begin{equation}
\label{KhovRel3}
\stik{1}{+ \& + \\
+ \ar{ur} \& +  \ar{ul} \\
+ \ar{ur} \& + \ar{ul}
} 
=
\ttstik{1}{
+ \& + \\
{} \& {} \\
+\ar{uu} \& + \ar{uu}
}
\stik{1}{}
,
\stik{1}{}
\nttstik{1}{+ \& + \& + \\
+ \ar{ur} \& +  \ar{ul}\& + \ar{u} \\
+ \ar{u} \& + \ar{ur}\& + \ar{ul}
\\
+ \ar{ur} \& +  \ar{ul}\& + \ar{u}
} 
=
\nttstik{1}{+ \& + \& + \\
+ \ar{u} \& + \ar{ur}\& + \ar{ul} \\
+ \ar{ur} \& +  \ar{ul}\& + \ar{u}
\\
+ \ar{u} \& + \ar{ur}\& + \ar{ul}
} 
\end{equation}

As we saw $m_{a_1}, \Delta_{b_1}$ and their compositions correspond to elements in $K(\KH')$, but to get objects corresponding to $m_{a_i},\Delta_{a_i}$ for $i\geq 2$ one needs to consider the Karoubi envelope of $\KH'$, which is denoted by $\KH$. Sending $m_{a_i}, \Delta_{b_i}$ to the classes of objects of $\KH$ given by the anti-symmetrization and symmetrization idempotents in $\End(+^i)$ and $\End(-^i)$ defines a map $\gamma:\Heis(\Lambda)\rightarrow K_0(\KH)$. Denote these objects by $m_{\Lambda^i}$ and $\Delta_{S^i}$ (Thus, in particular we now denote $+$ and $-$ by $m_{\Lambda^1}$,$\Delta_{S^1}$). To show that $\gamma$ is a morphism of algebras we need to construct for all $i$ the isomorphisms 
\begin{align}
\label{khovdeltam}
\Delta_{S^i}m_{\Lambda^i}\cong m_{\Lambda^i} \Delta_{S^i}\bigoplus m_{\Lambda^{i-1}}\Delta_{S^{i-1}}
\end{align}

For $i=1$ $S^1=\Lambda^1=\Id$. Also $m_{\Lambda^0}$ is an identity map. Hence above is just the isomorphism 
\begin{equation}
\label{KhovIso}
\Delta_{\Id}m_{\Id}\cong m_{\Id} \Delta_{\Id}\bigoplus \Id
\end{equation}
The planar diagram representing this isomorphism is easy to construct using the graphical calculus for $\KH'$ that we just described. Specifically, it is equivalent to the rows \eqref{KhovRel1} and \eqref{KhovRel2} in the list. The construction for $i\geq 2$ involves the more complicated graphical calculus for the Karoubi envelope $\KH$ whose existence is implied in \cite{Khov}. There it is further shown that $\gamma$ is injective. That $\gamma$ is in fact an isomorphism is a conjecture.

\section{Categorical Heisenberg double and Khovanov's graphical calculus}
\label{KhovCalculus}
Let $\CCC$ be a $k$-linear semisimple abelian category with a symmetric selfadjoint Hopf structure in the sense of \autoref{HopfCategory}. Through this section we will assume that the left and right adjoints of the multiplication functor $m:\CCC\otimes\CCC\rightarrow\CCC$ are equal, and denote $\Delta^r=\Delta^l$ just by $\Delta$. We would like to relate our categorfication of the Heisenberg algebra via the Heisenberg double construction to the categorification of Khovanov in \cite{Khov}. 

The first thing we note is that our categorification doesn't depend on the choice of adjunction data. For any choice of units and counits of adjunction we can construct the canonical isomorphism from \autoref{th:deltam}. In the construction of \cite{Khov} one has the first relation of \eqref{KhovRel1} between the left unit and the right counit of adjunction. In what follows we will carefully analyze the construction of the isomorphism from \autoref{th:deltam} to prove the following statement
\begin{Theorem}
Given adjunction data satisfying relations \eqref{KhovRel1} for an SSH category $\CCC$ 
we can construct a faithful functor \[A:\KH\rightarrow\EndAd{}{\CCC}\] 
where $\EndAd{\CCC}$ of is the category of endofunctors admitting adjoints. 
\end{Theorem}
Moreover such a functor can be constructed for every irreducible \emph{primitive} element of $\CCC$. Denote the image of $\kk \in \Vect$ under the unit functor $m_{\emptyset}:\Vect \rightarrow \CCC$ by $\one$ (see \autoref{SSHunit}).
\begin{Definition}
An object $P$ of $\CCC$ is called \emph{primitive} if $\Delta P\cong P \boxtimes \one \oplus \one\boxtimes P $ 
\end{Definition}
This definition is motivated by the notion of  primitive element in a Hopf algebra.

We would like to construct a functor $A:\KH\rightarrow\EndAd{}{\CCC}$ such that $A(+)=m_P$, $A(-)=\Delta_P$ for $P$ - an irreducible primitive object in $\CCC$. Recalling the construction of $\KH$ from \autoref{KhovanovHeisenberg} we immediately see that it is enough to define the functor $F$ on the category $\KH'$ such that $\KH$ is its Karoubi envelope. Therefore the value of $A$ on the objects of $\KH$ is defined by our condition and it remains to define $A$ on morphisms. $\Delta_P$ is left-and right-adjoint to $m_P$. The diagrams representing identity morphisms and units and counits of adjunctions in $\KH'$ \eqref{diag:unitcounit} should go to the corresponding morphisms in $\EndAd{\CCC}$. Denote the images of these by $\eta_L, \epsilon_L$ and $\eta_R, \epsilon_R$ respectively.

Since all crosses are mates by \autoref{crossmates} it remains to specify an image of one of them. The SSH category $\CCC$ is symmetric monoidal (Proposition 3.18 of \cite{mySSH}). Let the upward cross \eqref{diag:ucross} go to the isomorphism $\beta:m_Pm_P\rightarrow m_Pm_P$ arising from the symmetric monoidal structure on $\CCC$. We immediately see the relations \eqref{KhovRel3} follow from the fact that $\beta$ is a symmetric braiding. 

It remains to verify the relations \eqref{KhovRel1} and \eqref{KhovRel2}. In Khovanov's construction these relations are equivalent to the isomorphism \eqref{KhovIso}. Namely the splitting \eqref{KhovIso} in this construction is given by the pairs of morphisms
\begin{equation*}
\stik{1}{- \ar[bend right=90]{r} \& + \& - \& + \ar[bend right=90]{l} \&,\&}
\stik{1}{}
\stik{1}{- \& + \\ + \ar{ur} \& -  \ar[leftarrow]{ul}} \stik{1}{}
\stik{1}{+ \& - \\ - \ar[leftarrow]{ur} \& +  \ar{ul}} \stik{1}{}
\end{equation*}

$F$ sends these morphisms to the pairs $\eta_L, \epsilon_R$
and the left and right mates of $\beta$: $\epsilon_L\beta\eta_L$ and $\epsilon_R\beta\eta_R$. 

Recall that by \autoref{th:deltam} the self-adjoint Hopf structure on $\CCC$ implies for any pair of objects $X,Y \in \CCC$ the existence of a canonical isomorphism 
\begin{equation}
\label{HeisIso}
\Delta^r_Xm_Y\cong m\circ (\Delta^r)^2_{\Delta(X)}\circ i_Y
\end{equation}
Setting $X=P, Y=P$ in \autoref{HeisIso} we get 
\begin{equation}
\label{Delta2Primitive}
\Delta^2_{\Delta P}\circ i_P=(\Delta)^2_{\one\boxtimes P \oplus P\boxtimes\one}\circ i_P=P\boxtimes\Delta_P\oplus\one\boxtimes\Id\end{equation}
Hence we obtain:
\begin{equation}
\label{HeisSplitting}
\Delta_Pm_P\cong m\circ (\Delta)^2_{\Delta(P)}\circ i_P\cong m\circ(\Delta)^2_{P\boxtimes\one\oplus\one\boxtimes P}\circ i_P=m_P\Delta_P\oplus\Id
\end{equation}
This is precisely the splitting of the form \eqref{KhovIso}.
\begin{Theorem}
\label{KhovanovAction}
For an irreducible primitive object $P$ of an SSH category $\CCC$ the isomorphism \autoref{HeisSplitting} is equivalent to relations \eqref{KhovRel1} and \eqref{KhovRel2} iff the left counit of adjunction $\epsilon_L:\Delta_P m_P\cong m_P \Delta_P\oplus\Id\rightarrow \Id$ is the inverse of $\eta_R$ on the $Id$ component and $0$ on $m_P\Delta_P$.
\end{Theorem}
\begin{proof}
To check that the relations \eqref{KhovRel1} and \eqref{KhovRel2} hold it remains to show that the splitting \eqref{HeisSplitting} coming from the Hopf structure on $\CCC$ is given by the pairs $\eta_L, \epsilon_R$
and the left and right mates of $\beta$: $\epsilon_L\beta\eta_L$ and $\epsilon_R\beta\eta_R$ for the primitive element $P$ of $\CCC$

The isomorphism \eqref{HeisSplitting} is obtained by taking the right mate of the square:
\begin{equation}
\tik
\CS{2} \arrow[]{d}[left]{m^2_{\Delta(P)}} \arrow[]{r}[above]{m} & \CCC \arrow[shorten >=0.3cm,shorten <=0.3cm,Rightarrow]{dl}[above,sloped]{\alpha}  \arrow[]{d}[right]{m_P}\\
\CS{2} \arrow[]{r}[below]{m}& \CCC 
\tak
\end{equation}
Explicitly, it is the following composition:
\begin{align}
\label{HeisIsoDet}
\begin{split}
m_P\Delta_P\oplus\Id&=m(\Delta)^2_{\Delta(P)} i_P\\ &\xrightarrow{\eta_R}\Delta_Pm_Pm(\Delta)^2_{\Delta(P)}i_P\\
&\xrightarrow{\alpha} \Delta_Pmm^2_{\Delta(P)}(\Delta)^2_{\Delta(P)}i_P\\
&\xrightarrow{\epsilon^{\Delta(P)}_R} \Delta_Pmi_P
=\Delta_Pm_P
\end{split}
\end{align}
where $\epsilon_R^{\Delta(P)}$ is the counit of the adjunction $m^2_{\Delta (P)}\dashv\Delta^2_{\Delta (P)}$.

Since $P$ is primitive we can expand the third line of this sequence of morphisms:
\begin{align}
\label{3line}
\begin{split}
\Delta_Pmm^2_{\Delta_P}(\Delta)^2_{\Delta(P)}i_P=\Delta_Pm\biggl(&m(\one\boxtimes P)\boxtimes m(P\boxtimes\Delta_P)\\&\oplus m(\one\boxtimes\one)\boxtimes m(P\boxtimes\Id) \\
&\oplus m(P\boxtimes P)\boxtimes m(\one\boxtimes\Delta_P)\\
&\oplus m(P\boxtimes\one)\boxtimes m(\one\boxtimes\Id)\biggr)
\end{split}
\end{align}
\begin{Lemma}
\label{lemmaepsilon}
$\epsilon_R^{\Delta(P)}$ is zero on the second and third component of this direct sum. Its restriction on the first component is given by $\epsilon_P^R$ and restriction on the fourth component is $c\Id$ where $c$ is a constant.  
\end{Lemma}
Looking at the second line of \autoref{HeisIsoDet}:
\begin{align*}
\Delta_Pm_Pm(\Delta)^2_{\Delta(P)}i_P=
\Delta_Pm_Pm(P\boxtimes\Delta_P\oplus\one\boxtimes\Id)
\end{align*}
we see that the image of $\Id$ under $\eta_R$ is contained in the second component of this direct sum. Hence
$\alpha\eta_R\left(\Id\right)$ is contained in the sum of the second and fourth components of \autoref{3line}. By \autoref{lemmaepsilon} we have  $\epsilon_R^{\Delta}\alpha\eta_R\left(\Id\right)=\eta_R\left(\Id\right)$ (multiplying $\epsilon_R^{\Delta(P)}$ by a constant if needed). 

Likewise, as $\alpha\eta_R\left(m_P\Delta_P\right)$ is contained in the sum of the first and third components and the restriction of $\epsilon_R^{\Delta(P)}$ to the first component is equal to $\epsilon_R$ by \autoref{lemmaepsilon} we obtain $\epsilon_R^{\Delta (P)}\alpha\eta_R\left(m_P\Delta_P\right)=\epsilon_R\beta\eta_R$. Here $\beta:m_Pm_P\rightarrow m_Pm_P$ is the braiding and from the decomposition of $\alpha$ from \autoref{HeisDecomp} it follows that the restriction of $\alpha$ to the image of $m_P\Delta_P$ is equal to $\beta$. Hence the injection morphisms for the splitting \eqref{HeisSplitting} are given by the images of the injection morphisms in $\KH'$. 

The following then follows immediately from analyzing the relations in $\KH'$:
\begin{Proposition}
\label{epsilonL}
The relations
\[
\stik{1}{
- \ar[bend right=90]{r} \& + \ar[bend right=90]{l}
}
= \Id_{\monunit}
\stik{1}{}
,
\stik{1}{}
\stik{1}{
{} \& {} \& + \\
- \& + \ar[bend right=90]{l} \& + \ar{u} \\
- \ar[bend right=90]{r} \ar[leftarrow]{u} \& + \ar{ur} \& + \ar{ul}\\
{} \& {} \& + \ar{u}
}
=0
\]
between the morphisms in $\KH'$ are satisfied between their images under $A$ iff we have that $\epsilon_L:\Delta_P m_P\cong \Id \oplus m_P \Delta_P\rightarrow \Id$ is the inverse of $\eta_R$ on the $\Id$ component and $0$ on $m_P\Delta_P$. 
\end{Proposition}

Now let us check that the relation
\begin{equation}
\label{rightleftrelation}
\stik{1}{+ \& - \\
- \ar[leftarrow]{ur} \& +  \ar{ul} \\
+ \ar{ur} \& - \ar[leftarrow]{ul}
} 
=
\ttstik{1}{
+ \& - \\
{} \& {} \\
+\ar{uu} \& - \ar[leftarrow]{uu}
}
\end{equation}
is satisfied for such left adjunction data.

Let us write out the image of \autoref{rightleftrelation} under $A$:
\begin{align*}
\begin{split}
m_P \Delta_P&\xrightarrow{\epsilon_R\beta\eta_R}\\
&\Delta_P m_P\cong\Id \oplus m_P\Delta_P \xrightarrow{\eta_L} \\ 
&(\Id \oplus m_P\Delta_P)m_P\Delta_P\cong m_P\Delta_P\oplus m_P\Delta_P \oplus m_P^2\Delta_P^2 \xrightarrow{\beta}\\ 
& m_P\Delta_P\oplus m_P\Delta_P \oplus m_P^2\Delta_P^2 \xrightarrow{\epsilon_L}\\ 
& m_P \Delta_P
\end{split}
\end{align*}
We need to check that this composition is equal to identity. As we saw above $\epsilon_R\beta\eta_R$ takes $m_P\Delta_P$ to $m_P\Delta_P$ in the direct sum decomposition of $\Delta_P m_P$. Next composing $\eta_L$ with this we get 
\begin{align}
\begin{split}
\label{etaL}
&\Id\oplus m_P\Delta_P\xrightarrow{\eta_L} \\&m_P\Delta_P\oplus m_P\Delta_P m_P\Delta_P \cong m_P(\Id\oplus m_P\Delta_P)\Delta_P \cong m_P\Delta_P\oplus m_P\Delta_P \oplus m_P^2\Delta_P^2
\end{split}
\end{align}
From the definition of $\epsilon_L$ in \autoref{epsilonL} and the unit-counit relation 
\begin{align}
  \Delta_P &\xrightarrow{\Delta_P\eta_L}\Delta_P m_P \Delta_P\cong (\Id\oplus m_P\Delta_P)\Delta_P\xrightarrow{\epsilon_L\Delta_P} \Delta_P  
\end{align}
we deduce that the composition of $\eta_L$ with projection to the second $m_P\Delta_P$ in the direct sum decomposition is $\Id$.
Next, $\beta$ switches the first and second copies of $m_P\Delta_P$ in the direct sum decomposition.
Applying $\epsilon_L$ we get 
\begin{align}
m_P\Delta_P\oplus m_P\Delta_P \oplus m_P^2\Delta_P^2\cong(\Id\oplus m_P\Delta_P) m_P\Delta_P\xrightarrow{\epsilon_L}m_P\Delta_P
\end{align}
where $\epsilon_L$ by definition is the projection on the $\Id$ component.
Altogether we have shown that the composition is identity. 

The last remaining relation 
\[
\stik{1}{- \& + \\
+ \ar{ur} \& -  \ar[leftarrow]{ul} \\
- \ar[leftarrow]{ur} \& + \ar{ul}
} 
=
\ttstik{1}{- \& + \\
{}\& {} \\
- \ar[leftarrow]{uu} \& + \ar{uu}}
-
\ttstik{1}{- \ar[bend right=90]{r}\& + \\
{} \& {} \\
-  \& + \ar[bend right=90]{l}
} 
\]
then follows immediately. 
\end{proof}
%\begin{Conjecture}
%$\epsilon_L$ is part of adjunction data.
%\end{Conjecture}
%This implies
%\begin{Conjecture}
%Let $\CCC$ be a symmetric self-adjoint Hopf category with associated Heisenberg double action \lena{Is that how we call it?}. Then we can construct a faithful functor \[F:\KH\rightarrow\EndAd{}{\CCC}\] 
%\end{Conjecture}
\begin{proof}[Proof of the \autoref{lemmaepsilon}]
We have for every $G\boxtimes F$:
\begin{align}
\begin{split}
 m^2_{\Delta_P}(\Delta)^2_{\Delta(P)}G\boxtimes F=&G\boxtimes m(P\boxtimes\Delta_P(F))\\&\oplus m(P\boxtimes G)\boxtimes \Delta_P(F) \\
&\oplus \Delta_P(G)\boxtimes m(P\boxtimes F)\\&\oplus m(P\boxtimes \Delta_PG)\boxtimes F
\end{split}
\end{align}
$\epsilon_R^{\Delta(P)}$ maps this to $G\boxtimes F$ and is a grade preserving morphism. Hence it is non-zero only the first and last components. There it is given by $\Id\boxtimes\epsilon_R$ and $\epsilon_R\boxtimes\Id$ respectively.

If we set $G=P$, so that $m_P\Delta_P(G)=m(P\boxtimes\Id)=P$ the restriction to the fourth component is $\epsilon_R\boxtimes \Id: P\boxtimes F \rightarrow P\boxtimes F$ hence it is $c\Id\boxtimes \Id$ since $P$ is irreducible.
\end{proof}
\begin{Corollary}
\label{KhovanovActionEpsilonL}
Consider an SSH category $\CCC$ with an irreducible primitive element $P$. Then we have the decomposition $\Delta_P m_P\cong \Id \oplus m_P \Delta_P$ with $\eta_R$ being the embedding $\Id \rightarrow \Id\oplus m\Delta$.  Define the functor $\epsilon_L$  by setting it to be the inverse of $\eta_R$ on the $\Id$ component and $0$ on $m_P\Delta_P$. Then if $\epsilon_L$ is the counit of adjunction, i.e. if it defines an isomorphism for all $X,Y \in \CCC$
\begin{align*}
\Hom(Y,mX) \rightarrow Hom(\Delta Y,X)
\end{align*}
We recover a strong categorical Heisenberg action on $\CCC$ in the sense of \cite{Khov}.
\end{Corollary}
\begin{proof}
The above definition of $\epsilon_L$ immediately implies \autoref{KhovanovAction}.
\end{proof}

\begin{Example}
The only irreducible primitive element in the category of polynomial functors $\PP$ is the identity functor. Our construction for this element and the SSH structure on $\PP$ from \cite{mySSH} recovers the categorical action of Khovanov's category $\KH$ constructed in \cite{hongyacobiPoly}. 

\end{Example}

\begin{Remark}
\autoref{KhovanovActionEpsilonL} provides a recipe for deriving graphical calculus from the canonical splitting \autoref{HeisSplitting} provided by the symmetric selfadjoint Hopf structure. The right unit $\eta_R$ and the splitting \autoref{HeisSplitting} define the rest of the adjunction data. In particular one can write the formula for the Casimir element.
\end{Remark}

\printbibliography
\end{document}